\documentclass[10]{amsart}
\usepackage{amsmath,amsfonts,amssymb, amsthm}
\usepackage{bbold}
\usepackage[margin=1.5in]{geometry}

\usepackage{amssymb, amsmath, mathptmx}
\usepackage{cite}

\newtheorem{theorem}{Theorem}
\newtheorem{proposition}{Proposition}[section]
\newtheorem{lemma}[proposition]{Lemma}
\newtheorem{corollary}[proposition]{Corollary}

\title{Cycle Index Sum for Non-$k$-Equal Configurations}
\author{Keely Grossnickle and Victor Turchin}
\address{Department of Mathematics\\
  Kansas State University\\
  Manhatan, KS 66506, USA}
  \email{kgrossni@ksu.edu, turchin@ksu.edu}
\date{}
\begin{document}

\maketitle
\sloppy

\begin{abstract}
We compute the cycle index sum of the symmetric group action on the homology of the configuration spaces of points in a Euclidean space with the condition that no $k$ of them are equal.
\end{abstract}

\section{Introduction} \label{section1}

Let $\mathcal{M}_{d}^{(k)}(n)$ be the configuration space of $n$ labeled points in $\mathbb{R}^{d}$ with the $non-k-equal$ $condition$: no $k$ points coincide. For example, $\mathcal{M}_{d}^{(2)}(n)$ is the usual configuration space of $n$ distinct points in~$\mathbb{R}^{d}$. Bj\"{o}rner and Welker in \cite{B_W} first computed the homology of $\mathcal{M}_{d}^{(k)}(n)$ for $k \geq 3$. Sundaram and Wachs in \cite{S_W} later computed the symmetric group action on the homology of the intersection lattice corresponding to $\mathcal{M}_{d}^{(k)}(n)$; their computations imply the following isomorphism of symmetric sequences:

\begin{equation}
H_{\ast}\mathcal{M}_{d}^{(k)} \simeq Com \circ (\mathbb{1} \oplus (\mathcal{L}ie \circ \mathcal{H}_{1}^{(k)}) \{d-1\} ). \tag{1.1} \label{eq:iso}
\end{equation}
where $\circ$ is the graded composition product for symmetric sequences \cite[Section 2.2.2]{fresse}.\footnote{Explicitly this formula appears in \cite{D_T}} Recall also that $H_{\ast}\mathcal{M}_{d}^{(k)}(n)$ is torsion free. The isomorphism $\eqref{eq:iso}$ holds integrally for $d \geq 2$, $k\geq 3$ and rationally for $d\geq 1$, $k \geq 2 $.  $Com$ and $\mathcal{L}ie$ are the underlying symmetric sequences of the commutative and Lie operads and $\mathcal{H}_{1}^{(k)}$ is the symmetric sequence of hook representations that we describe in section~3. The notation $\{d-1\}$ is the operadic degree $(d-1)$ suspension of symmetric sequences. The symmetric sequence $\mathbb{1}$ is the unit with respect to the composition product. It is a one dimensional space concentrated in arity 1. The cycle index sum of the symmetric group action on $H_{\ast}\mathcal{M}_{d}^{(2)}$, the usual configuration space, was computed in \cite{lehrer, A_T} to be:

\begin{equation}
Z_{H_{\ast}\mathcal{M}_{d}^{(2)}} = \prod_{m=1}^{\infty}\left(1+(-1)^{d}(-q)^{m(d-1)}p_{m}\right)^{(-1)^{d}E_{m}\left(\frac{1}{(-q)^{d-1}}\right)} . \tag{1.2} \label{k2 iso}
\end{equation}

From \eqref{eq:iso}, in \cite{D_T} the exponential generating function of Poncair\'e polynomials for the sequence $H_{\ast}\mathcal{M}_{d}^{(k)}$ is computed to be: 

\begin{equation}
F_{H_{\ast}\mathcal{M}_{d}^{(k)}}(x)=\sum_{n=0}^{\infty}P_{H_{\ast}\mathcal{M}_{d}^{(k)}(n)}(q)\frac{x^{n}}{n!}=e^{x}\left(1-(-q)^{k-2}+(-q)^{k-2}\left(\sum_{j=0}^{k-1}\frac{(-q^{d-1}x)^{j}}{j!}\right)e^{q^{d-1}x} \right)^{-\frac{1}{q^{d-1}}} .\tag{1.3} \label{dim iso}
\end{equation}

The main result of this paper describes the cycle index sum of the symmetric sequence $H_{\ast}\mathcal{M}_{d}^{(k)}$ obtained from the isomorphism \eqref{eq:iso}.

\begin{theorem} For $k \geq 2$, $d\geq 1$
\begin{multline}\label{equ:bigthm} \tag{1.4}
Z_{H_{\ast}\mathcal{M}_{d}^{(k)}}(q;p_{1}, p_{2}, p_{3},...) = \\
=e^{(\sum_{l=1}^{\infty}\frac{p_l}{l})}\prod_{m=1}^{\infty}\Bigg(1-(-q)^{m(k-2)}+\\(-q)^{m(k-2)}\Big(e^{-\sum_{j=1}^{\infty}\frac{(-1)^{d-1}(-q)^{mj(d-1)}p_{mj}}{j}}\Big)_{\leq m(k-1)}\Big(e^{\sum_{j=1}^{\infty}\frac{(-1)^{d-1}(-q)^{mj(d-1)}p_{mj}}{j}}\Big)\Bigg)^{(-1)^{d}E_{m}\Big(\frac{1}{(-q)^{d-1}}\Big)} ,
\end{multline} 
where $\leq m(k-1)$ denotes the truncation with respect to the cardinality degree ($|p_i|=i$) and $E_{m}(y)=\frac{1}{m}\sum_{i\mid m}\mu(i)y^{\frac{m}{i}}$, where $\mu(i)$ is the usual M\"{o}bius function.
\end{theorem}
\noindent Most of the computations are straightforward. The main difficult part is computing the cycle index sum for $\mathcal{H}_{1}^{(k)}$, which is done in Section 3. 

It is easy to see that from \eqref{equ:bigthm} one can recover \eqref{k2 iso} by setting $k=2$. Similarly, \eqref{dim iso} can be recovered from \eqref{equ:bigthm} by setting $p_{1}=x$ and $p_{i}=0$ for $i\geq 2$.

We also establish a refinement of Theorem 1. The homology groups of $H_{\ast}\mathcal{M}_{d}^{(k)}$ can be described as linear combinations of certain products of iterated brackets\cite{D_T}. These brackets are of two types: long or short. The number of long, respectively short, brackets represent two additional gradings on the space. The cycle index sum of $H_{\ast}\mathcal{M}_{d}^{(k)}$ can be adjusted with the use of two additional variables to account for these two additional gradings. See Theorem 2 in Section 5.

The formula \eqref{k2 iso} was used in \cite{A_T, ST_T, T} to compute the generating functions for the Euler characteristics of the terms of the Hodge splitting in the rational homology of the spaces of higher dimensional long knots and string links. In the same way, the results of this paper can be used to compute the Euler characteristics of the Hodge splitting in the second term of the Goodwillie-Weiss or Vassiliev spectral sequences for spaces of long non-$k$-equal (string) immersions\cite{D_T}. The differential $d_{1}$ of the above spectral sequences preserves the number of long and short brackets used in the refinement, which is our motivations for Theorem 2.

\subsection*{Acknowledgments}
The authors are thankful to Fr\'{e}d\'{e}ric Chapoton, Vladimir Dotsenko and Anton Khoroshkin for communication.

\section{Notation and Basic Facts about the Cycle Index Sum} \label{section 2}

In this paper we will use the variable $q$ for the formal variable responsible for the homological degree. For $\sigma \in \Sigma_{n}$ we will denote the number of its cycles of length $j$ by $\ell_{j}(\sigma)$. Let $\rho:\Sigma_{n}\rightarrow GL(V)$ be a representation of the symmetric group $\Sigma_{n}$, where $V$ is a graded vector space, and let $(p_{1}, p_{2}, p_{3},...)$ be a family of infinite commuting variables. Then the cycle index sum of $\rho$, denoted $Z_{V}(q;p_{1}, p_{2}, p_{3},...)$, is defined by
\begin{equation} \label{def:cycle}
Z_{V}(q;p_{1}, p_{2}, p_{3},....) = \frac{1}{|\Sigma_{n} |} \sum_{\sigma \in \Sigma_{n}}tr(\rho(\sigma))\prod_{j}p_{j}^{\ell_{j}(\sigma)},\tag{2.1}
\end{equation}
where $tr(\rho(\sigma))$ is the graded trace that is a polynomial of $q$ obtained as the generating function of traces on each component. There is also an auxiliary \textit{cardinality degree} given by $p_{i}$'s where each $p_{i}$ is said to have cardinality degree $i$. Below we recall some facts about the cycle index sum.

Let $V$ be a $\Sigma_{k}$-module and $W$ be a $\Sigma_{n}$-module. Then from \cite[section 6.1; section 3.1, proposition 8, part c; 7.3, respectively]{bergeron, B_T_T, macdonald},
\begin{equation}
Z_{\mathrm{Ind}_{\Sigma_{k}\times\Sigma_{n}}^{\Sigma_{k+n}}(V \otimes W)} = Z_{V} \cdot Z_{W}. \tag{2.2} \label{cis:prod}
\end{equation}
For a symmetric sequence $M(\bullet) = \{M(n), n \geq 0\}$, one defines its cycle index sum as
\begin{equation}
Z_{M}(q;p_{1}, p_{2}, p_{3},...)=\sum_{n=0}^{\infty}Z_{M(n)}(q;p_{1}, p_{2}, p_{3},...) .\tag{2.3} \label{cis:sum}
\end{equation}
For the proof of Theorem 1, we will need the formula for the graded plethysm:

\begin{equation}
Z_{M \circ N} = Z_{M} \ast Z_{N} = Z_{M}(q;p_{i} \mapsto p_{i} \ast Z_{N}) ,\tag{2.4.1}\label{plethysm:1}
\end{equation}
where
\begin{equation}
p_{i}\ast Z_{N} = Z_{N}(q \mapsto (-1)^{i-1}q^{i}; p_{j}\mapsto p_{ij}). \tag{2.4.2}\label{plethysm:2}
\end{equation}

The usual plethysm without the grading can be found in \cite[equation 3.25, section 3.8; definition 3, section 1.4; equation 8.1-8.2, section 8, respectively]{bergeron, B_T_T, macdonald}. For the graded case, it is done when $q=-1$ in \cite[section 7.20]{G_K}. Unfortunately, the graded version of this formula doesn't seem to appear in the literature, though it is known to experts, \cite{C_D_K, mathoverflow}. To prove our formula, we notice that the sign convention is correct and holds when $q=-1$ and the $q$-grading contribution is correct by the same argument as in \cite[Section 3.5, definition 3]{D_K}.

To recall the operadic suspension $\mathcal{M}\{1\}$ of the symmetric sequence $\mathcal{M}$ is defined as $$\mathcal{M}\{1\}(n) = s^{n-1}\mathcal{M}(n)\otimes V_{(1^{n})},$$ where $s^{n-1}$ is the degree $(n-1)$ suspension and $V_{(1^{n})}$ is the sign representation. One can easily see that $$Z_{\mathcal{M}\{1\}} = \frac{1}{q}Z_{\mathcal{M}}(q;p_{i}\mapsto (-1)^{i-1}q^{i}p_{i}).$$ We will use the formula for the $\{d-1\}$ operadic suspension, which is an easy formula to obtain from the above: 
\begin{equation}
Z_{\mathcal{M}\{d-1\}}=(q)^{1-d}Z_{\mathcal{M}}(q;p_{i}\mapsto (-1)^{(i-1)(d-1)}q^{i(d-1)}p_{i}). \tag{2.5} \label{graded:susp}
\end{equation}

Lastly, we will need the cycle index sums of $Com$ and $\mathcal{L}ie$. From \cite{D_K, G_K}, the cycle index sum for $Com$ is
\begin{equation}
Z_{Com} = \exp\left(\sum_{i=1}^{\infty} \frac{p_{i}}{i}\right); \tag{2.6}\label{cis:com}
\end{equation}
and from \cite{brandt, D_K, G_K} , the cycle index sum for $\mathcal{L}ie$ is
\begin{equation}
Z_{\mathcal{L}ie} = \sum_{i=1}^{\infty}\frac{-\mu(i)\ln(1-p_{i})}{i}, \tag{2.7}\label{cis:Lie}
\end{equation}
where, as before and throughout this paper, $\mu(i)$ is the usual M\"{o}bius function. 

We will also use the notation $V_{\lambda}$ to denote the irreducible $\Sigma_{n}$-representation corresponding to the partition $\lambda$, see \cite{fulton}.
\section{Cycle Index Sum for the Sequence of Hooks $\mathcal{H}_{1}^{(k)}$} \label{section3}
We define $\mathcal{H}_{1}^{(k)}(n)$ as a graded $\Sigma_{n}$ module, which is trivial if $n<k$ and
$\mathcal{H}_{1}^{(k)}(n) = s^{k-2}V_{(n-k+1, 1^{k-1})}$ otherwise, where $V_{(n-k+1, 1^{k-1})}$ is the hook representation corresponding to the partition $\lambda = (n-k+1, 1^{k-1})$ and $s^{k-2}$ is the $(k-2)$-suspension. The space $\mathcal{H}_{1}^{(k)}(n)$ is  some natural subspace of $H_{k-2}\mathcal{M}_{1}^{(k)}(n)$, which explains why $\mathcal{H}_{1}^{(k)}(n)$ lies in degree $k-2$, see \cite{D_T}.

\begin{proposition} \label{prop3.1}

For $k \geq 2$,

\begin{equation}
Z_{\mathcal{H}_{1}^{(k)}}(q;p_{1}, p_{2}, p_{3}, ... )=(-q)^{k-2} - (-q)^{k-2}\left(\exp\left(-\sum_{i=1}^{\infty}\frac{p_{i}}{i}\right)\right)_{\leq k-1}\left(\exp\left(\sum_{i=1}^{\infty}\frac{p_{i}}{i}\right)\right), \tag{3.1}\label{hookprop}
\end{equation}
where $\leq k-1$ is the truncation with respect to the cardinality degree.
	
\end{proposition}
\noindent  We will prove this proposition with the following well-known facts and lemmas. First, let $W_{n}=\mathbb{Q}[\underline{n}]$, where $\underline{n}=\{1,2,...,n\}$, be the canonical $n$-dimensional representation of $\Sigma_{n}$. Then $W_{n}$ can be decomposed in the following way: $W_{n}=V_{(n-1,1)}\oplus V_{(n)}$ where $V_{(n-1,1)}$ is the $(n-1)$ dimensional representation and $V_{(n)} = \mathbb{Q}$ is the one-dimensional trivial representation. 

\begin{lemma} \label{lemma3.2}
For $n\geq k \geq 0$,
$\wedge^{k}W_{n}= \mathrm{Ind}_{\Sigma_{k}\times \Sigma_{n-k}}^{\Sigma_{n}} V_{(1^{k})} \otimes V_{(n-k)}$.

\end{lemma}

\begin{proof}
First recall that $V_{(1^{k})}$ is the sign representation of $ \Sigma_{k}$ and that $V_{(n-k)}$ is the trivial representation of $ \Sigma_{n-k} $. Also, note that $\wedge^{k} W$ and $\mathrm{Ind}_{\Sigma_{k}\times \Sigma_{n-k}}^{\Sigma_{n}} V_{(1^{k})} \otimes V_{(n-k)}$ have the same dimension, namely $\binom{n}{k}$.
To start, let $e_{1}, e_{2},...,e_{n}$ be the usual basis of $W_{n}$. Now examine how $\Sigma_{n}$ acts on a vector $e_{i_{1}}\wedge e_{i_{2}}\wedge ... \wedge e_{i_{k}} \in W_{n}$. For $\sigma \in \Sigma_{n}$, $\sigma(e_{i_{1}}\wedge e_{i_{2}}\wedge ... \wedge e_{i_{k}}) = e_{\sigma(i_{1})}\wedge e_{\sigma(i_{2})}\wedge ... \wedge e_{\sigma(i_{k})}$. By definition, $\mathrm{Ind}_{\Sigma_{k}\times \Sigma_{n-k}}^{\Sigma_{n}} V_{(1^{k})} \otimes V_{(n-k)} = \mathbb{Q}[\Sigma_{n}] \otimes_{\mathbb{Q}[\Sigma_{k}\times \Sigma_{n-k}]} V_{(1^{k})} \otimes V_{(n-k)}$. Define $$I_{(k,n-k)}:\mathbb{Q}[\Sigma_{n}] \otimes_{\mathbb{Q}[\Sigma_{k}\times \Sigma_{n-k}]} V_{(1^{k})} \otimes V_{(n-k)} \rightarrow \wedge^{k} W_{n}$$ by $I_{(k,n-k)}(\sigma \otimes \mathbb{1}) \mapsto \sigma  (e_{1}\wedge ... \wedge e_{k}) = e_{\sigma(1)}\wedge e_{\sigma(2)}\wedge...\wedge e_{\sigma(k)}$. We claim this is the desired isomorphism. First, we will show that it is well defined. Let $(\alpha, \beta) \in \Sigma_{k} \times \Sigma_{n-k} $. Then $I_{(k,n-k)}(\sigma \cdot (\alpha, \beta) \otimes \mathbb{1}) = e_{\sigma(\alpha(1))} \wedge e_{\sigma(\alpha(2))} \wedge ... \wedge e_{\sigma(\alpha(k))} = (-1)^{\mid \alpha \mid} e_{\sigma(1)} \wedge e_{\sigma(2)} \wedge .... \wedge e_{\sigma(k)} = (-1)^{\mid \alpha \mid} \sigma \otimes \mathbb{1}$. On the other hand, $I(\sigma \otimes (\alpha,\beta) \cdot \mathbb{1}) = \sigma \otimes (-1)^{\mid \alpha \mid} \mathbb{1} = (-1)^{\mid \alpha \mid} \sigma \otimes \mathbb{1}$. Therefore $I$ is well defined. As previously mentioned, these two spaces have the same dimension and by construction $I_{(k,n-k)}$ is surjective and therefore $I_{(k,n-k)}$ is bijective. 

\end{proof}

\begin{lemma} \label{lemma3.3}
For $n > k$, one has an isomorphism of $\Sigma_{n}$-modules:
$V_{(n-k, 1^{k})}=\wedge^{k} V_{(n-1,1)}$.
\end{lemma}

This lemma is a standard exercise in representation theory \cite[Exercise 4.6]{fulton}. 

\begin{corollary} \label{cor3.4}
One has an isomorphism of $\Sigma_{n}$-modules:
$\wedge^{k}W_{n} = \wedge^{k} V_{(n-1, 1)} \oplus \wedge^{k-1} V_{(n-1, 1)}$.
\end{corollary}

\begin{proof}
$\wedge ^{k}W_{n} = \wedge^{k}(V_{(n-1,1)} \oplus V_{(n)}) = \wedge^{k}(V_{(n-1,1)}) \oplus \wedge^{k-1}(V_{(n-1,1)}) \otimes V_{(n)}$, where $V_{(n)}$ is just the trivial representation and thus we have our desired isomorphism.
\end{proof}

\begin{corollary} \label{3.5}
One has an isomorphism of virtual $\Sigma_{n}$-modules:
$\wedge^{k}V_{(n-1,1)} = \sum_{i=0}^{k}(-1)^{i}\wedge^{k-i}W_{n}$
\end{corollary}

\begin{proof}
$ \wedge^{k}V_{(n-1,1)} = \wedge^{k}W_{n} - \wedge^{k-1}V_{(n-1,1)}$ by Corollary 3.4. 
We apply the same corollary to $\wedge^{k-1}V_{(n-1,1)}$ again and we have $ \wedge^{k}V_{(n-1,1)} = \wedge^{k}W_{n} - \wedge^{k-1}V_{(n-1,1)} = \wedge^{k}W_{n} - \wedge^{k-1}W_{n} + \wedge^{k-2}V_{(n-1,1)}$. We can apply Corollary 3.4 iteratively to obtain the desired isomorphism.
\end{proof}

\noindent Now we are ready to prove Proposition 3.1.
\begin{proof}[Proof of Propsition 3.1]
Let \begin{equation*}
\mathcal{H}(n) = 
	\begin{cases}
	0, & n<k; \\
	V_{(n-k+1,1^{k-1})}, &\text{otherwise.}
	\end{cases}
\end{equation*}
In order to prove Proposition 3.1, it is sufficient to show that 
\begin{equation}
Z_{\mathcal{H}}(p_{1}, p_{2}, p_{3},...) = (-1)^{k-2} - (-1)^{k-2}\left(\exp\left(-\sum_{i=1}^{\infty}\frac{p_{i}}{i}\right)\right)_{\leq k-1}\left(\exp\left(\sum_{i=1}^{\infty}\frac{p_{i}}{i}\right)\right). \tag{3.2}\label{hooknodeg}
\end{equation}
For $n \geq k$, one has
\begin{align*}
\mathcal{H}(n) &= V_{(n-k+1,1^{k-1})}&&\text{(by definition)}\\
&= \wedge^{k-1}V_{(n-1,1)}&&\text{(by Lemma \ref{lemma3.3})}\\
&= \sum_{i=0}^{k-1}(-1)^{i}\wedge^{k-1-i}W_{n}&&\text{(by Corollary \ref{3.5})}\\
&= \sum_{i=0}^{k-1}(-1)^{i}\mathrm{Ind}_{\Sigma_{k-1-i} \times \Sigma_{n-k+1+i}}^{\Sigma_{n}}V_{(1^{k-1-i})}\otimes V_{(n-k+1+i)} &&\text{(by Lemma 3.2)}\\
&= (-1)^{k-1}\sum_{j=0}^{k-1} \mathrm{Ind}_{\Sigma_{j} \times \Sigma_{n-j}}^{\Sigma_{n}} V_{(1)^{j}} \otimes V_{(n-j)}. &&\text{(by taking $j = k-i-1$)}
\end{align*}
Next we apply \eqref{cis:prod}.

\begin{equation}
Z_{\mathcal{H}(n)}=Z_{V_{(n-k+1,1^{k-1})}} = (-1)^{k-1}\sum_{j=0}^{k-1}(-1)^{j}Z_{V_{(1^{j})}}\cdot Z_{V_{(n-j)}}. \tag{3.3} \label{3.3}
\end{equation}
Thus, 
\begin{equation}
Z_{\mathcal{H}}=(-1)^{k-1} \sum_{n\geq k} \sum_{j=0}^{k-1}(-1)^{j}Z_{V_{(1^{j})}} \cdot Z_{V_{(n-j)}}.\tag{3.4} \label{propproofequ}
\end{equation}
Note that $Z_{V_{(1^{j})}}$ is the cycle index sum for the sign representation and $Z_{V_{(n-j)}}$ is the cycle index sum for the trivial representation. We claim that  \eqref{propproofequ} is equal to \eqref{hooknodeg}. We will prove this claim in two cases: when cardinality $n<k$ and when cardinality $n \geq k$. 

We will first do the case when $n<k$. Clearly \eqref{propproofequ} is equal to 0 when $n<k$ as the sum starts when $n \geq k$ and thus has no terms. When $n<k$, \eqref{hooknodeg} is also 0 since the exponentials are inverses to one another:

\begin{multline*}
(-1)^{k-2} - (-1)^{k-2}\left(\exp\left(-\sum_{i=1}^{\infty}\frac{p_{i}}{i}\right)\right)_{\leq k-1}\left(\exp\left(\sum_{i=1}^{\infty}\frac{p_{i}}{i}\right)\right) =_{\leq k-1}\\ (-1)^{k-2} - (-1)^{k-2}\left(\exp\left(-\sum_{i=1}^{\infty}\frac{p_{i}}{i}\right)\right)\left(\exp\left(\sum_{i=1}^{\infty}\frac{p_{i}}{i}\right)\right) = 0.
\end{multline*}

\noindent Now we look at the case when the cardinality degree $n\geq k$. It follows from \eqref{cis:com} that $$\sum_{n=0}^{\infty}Z_{V_{(1^{n})}} = \exp\left(\sum_{i=0}^{\infty}\frac{(-1)^{i-1}p_{i}}{i}\right).$$ By replacing $p_{i} \mapsto (-1)^{i}p_{i}$, we get $$\sum_{n=0}^{\infty}(-1)^{n}Z_{V_{(1^{n})}} = \exp\left(\sum_{i=0}^{\infty}\frac{-p_{i}}{i}\right).$$ Then, $$\sum_{n=0}^{k-1}(-1)^{n}Z_{V_{(1^{n})}} = \exp\left(\sum_{i=0}^{\infty}\frac{-p_{i}}{i}\right)_{\leq k-1}.$$ We also know that $$\sum_{n=0}^{\infty}Z_{V_{(n)}} = \exp\left(\sum_{i=0}^{\infty}\frac{p_{i}}{i}\right).$$ From these formulas, one can easily see that in cardinality $n \geq k$, \eqref{hooknodeg} and \eqref{propproofequ} are both equal to \eqref{3.3}. Thus in arity $n$, \eqref{hooknodeg} is equal to \eqref{propproofequ}, completing the proof.
\end{proof}

\section{Proof of Theorem 1} \label{section4}

First we compute the plethsym of $\mathcal{L}$ie and $\mathcal{H}_{1}^{(k)}$ using \eqref{plethysm:1} and \eqref{plethysm:2}. 
\begin{multline}
Z_{\mathcal{L}ie \circ \mathcal{H}_{1}^{(k)}}(q;p_{1}, p_{2},p_{3},...)= \\
\sum_{i=1}^{\infty}\frac{-\mu(i)}{i}\ln\left(1-(-q)^{i(k-2)}+(-q)^{i(k-2)}\left[\exp\left(-\sum_{j=1}^{\infty}\frac{p_{ij}}{j}\right)\right]_{\leq i(k-1)}\left[\exp\left(\sum_{j=1}^{\infty}\frac{p_{ij}}{j}\right)\right]\right) .\tag{4.1}\label{lie:hook}
\end{multline}
Next we use \eqref{graded:susp} to compute the $\{d-1\}$ suspension of \eqref{lie:hook}:
\begin{multline}
Z_{(\mathcal{L}ie \circ \mathcal{H}_{1}^{(k)})\{d-1\}}(q;p_{1}, p_{2},p_{3},...)=\\ q^{1-d}\sum_{i=1}^{\infty} \frac{-\mu(i)}{i}\ln\Bigg(1-(-q)^{i(k-2)}+\\ (-q)^{i(k-2)}\bigg[e^{-\sum_{i=1}^{\infty}\frac{(-1)^{d-1}(-q)^{ij(d-1)}p_{ij}}{j}}\bigg]_{\leq i(k-1)}\bigg[e^{\sum_{i=1}^{\infty}\frac{(-1)^{d-1}(-q)^{ij(d-1)}p_{ij}}{j}}\bigg]\Bigg). \tag{4.2} \label{lie,hook,susp}
\end{multline}
Now we will simply add the $\mathbb{1}$, the trivial representation of $\Sigma_{1}$, to \eqref{lie,hook,susp}:

\begin{multline}
Z_{\mathbb{1}\oplus(\mathcal{L}ie \circ \mathcal{H}_{1}^{(k)})\{d-1\}}(q;p_{1}, p_{2},p_{3},...)=\\p_{1} + q^{1-d}\sum_{i=1}^{\infty} \frac{-\mu(i)}{i}\ln\Bigg(1-(-q)^{i(k-2)}+\\ (-q)^{i(k-2)}\bigg[e^{-\sum_{i=1}^{\infty}\frac{(-1)^{d-1}(-q)^{ij(d-1)}p_{ij}}{j}}\bigg]_{\leq i(k-1)}\bigg[e^{\sum_{i=1}^{\infty}\frac{(-1)^{d-1}(-q)^{ij(d-1)}p_{ij}}{j}}\bigg]\Bigg). \tag{4.3} \label{lie,hook,susp,id}
\end{multline}
Finally, we again use \eqref{plethysm:1} and \eqref{plethysm:2} to compute the graded composition product of $Com$ with \eqref{lie,hook,susp,id} to get an explicit formula.

\begin{multline}
Z_{Com \circ (\mathbb{1}\oplus(\mathcal{L}ie \circ \mathcal{H}_{1}^{(k)})\{d-1\})}(q;p_{1}, p_{2},p_{3},...)=\\ \exp\Bigg[\sum_{l=1}^{\infty}\frac{1}{l}\Bigg(p_{l}+(-1)^{d-1}(-q)^{l(1-d)}\sum_{i=1}^{\infty}\frac{-\mu(i)}{i}\ln\bigg(1-(-q)^{li(k-2)}\\+(-q)^{li(k-2)}\bigg[e^{-\sum_{j=1}^{\infty}\frac{(-1)^{d-1}(-q)^{lij(d-1)}p_{lij}}{j}}\bigg]_{\leq li(k-1)}\bigg[e^{\sum_{j=1}^{\infty}\frac{(-1)^{d-1}(-q)^{lij(d-1)}p_{lij}}{j}}\bigg]\bigg)\Bigg)\Bigg]. \tag{4.4}\label{com,lie,hook,susp}
\end{multline}
Recall that  $E_{m}(y) = \frac{1}{m} \sum_{i\mid m}\left(\mu(i)y^{\frac{m}{i}}\right)$. Using the substitution $m=li$ and the fact that the exponential function and the natural logarithm are inverses to one another, one obtains \eqref{equ:bigthm}.
\begin{flushright}
	$\Box$
\end{flushright}

\section{Refinement}

In \cite{D_T}, the homology groups $H_{\ast}\mathcal{M}_{d}^{(k)}(n)$ are described as linear combinations of certain products of iterated brackets, where there are two types of brackets, long and short. Long brackets have exactly $k$ inputs and cannot have any other brackets as elements inside them. Short bracket have exactly 2 inputs, either of which may be a long or short bracket.

$Example$: Let $n=7$, $k=3$. Two examples of homology classes one has are: $[\{x_{1},x_{3},x_{6}\},\{x_{2}, x_{4},x_{5}\}]\cdot x_{7}\in H_{5d-3}\mathcal{M}_{d}^{(3)}(7)$, and $[[[\{x_{1},x_{3},x_{5}\},x_{2}],x_{4}],x_{6}]\cdot x_{7} \in H_{5d-4}\mathcal{M}_{d}^{(3)}(7)$.

Geometrically, these classes can be viewed as products of spheres. For example, when $n=4$ and $k=3$, one homology class is $[\{x_{1}, x_{2}, x_{4}\},x_{3}] \in H_{3d-2}\mathcal{M}_{d}^{(3)}(4)$. This class is represented by $S^{2d-1}\times S^{d-1}$:

$$|x_{1}|^{2} + |x_{2}|^{2} + |x_{4}|^{2} = \epsilon^{2}, \quad 
x_{1}+x_{2}+x_{4}=0, \quad |x_{3}|^{2}=1, \quad (x_{1},x_{2},x_{3},x_{4}) \in (\mathbb{R}^{d})^{\times 4},$$ where $\epsilon \ll 1$.

The symmetric sequence $H_{\ast}\mathcal{M}_{d}^{(k)}$ has a left module structure over the homology of the little $d$-disks operad, which is the operad of Poisson algebras \cite{fresse}. The short bracket is the Lie operation in this operad.\footnote{For the notion of operad and left module over an operad, see for example \cite{fresse}}

The number of long and short brackets are additional gradings that we consider on $H_{\ast}\mathcal{M}_{d}^{(k)}(n)$. We add the variable $u$ to be responsible for the number of short brackets grading and the variable $w$ to be responsible for the number of long brackets grading in the graded trace used for the cycle index sum \eqref{def:cycle}. The sequence $Com$ does not contribute to these additional gradings and thus \eqref{cis:com} remains unchanged in the refinement. The graded suspension does not interact with the long and short brackets and thus \eqref{graded:susp} also remains unchanged. However, there are short brackets in $\mathcal{L}ie$. In cardinality k, there are always $k-1$ (short) brackets, which is why we divide by $u$ and replace $p_{i}$ by $u^{i}p_{i}$ in the formula below. By abuse of notation, we also denote by $Z_{\mathcal{L}ie}$ the cycle index sum of $\mathcal{L}ie$ with this refinement:
$$Z_{\mathcal{L}ie}(u,q;p_{1},p_{2},p_{3},...) = \sum_{i=1}^{\infty}\frac{-\mu(i)}{u}\frac{\ln(1-u^{i}p_{i})}{i}.$$

The space $\mathcal{H}_{1}^{(k)}(n)$ is a subspace of $H_{k-2}\mathcal{M}_{1}^{(k)}(n)$, defined as a subspace spanned by iterated brackets that have exactly one long bracket \cite{D_T}. This explains why we multiply by $w$ in the formula below. However, the iterated brackets of $\mathcal{H}_{1}^{(k)}(n)$ have exactly $n-k$ short brackets, which explains in the formula below why we divide by $u^{k}$ and replace $p_{i}$ by $u^{i}p_{i}$ in the refinement. Similarly we abuse notation to denote the cycle index sum of $\mathcal{H}_{1}^{(k)}(n)$ with the refinement as before by $Z_{\mathcal{H}_{1}^{(k)}}$.

\begin{multline*}
Z_{\mathcal{H}_{1}^{(k)}}(u,w,q;p_{1},p_{2},p_{3},...) =\\
 \frac{w}{u^{k}}\left((-q)^{k-2}-(-q)^{k-2}\left(\exp\left(-\sum_{i=1}^{\infty}\frac{u^{i}p_{i}}{i}\right)\right)_{ \leq k-1}\left(\exp\left(\sum_{i=1}^{\infty}\frac{u^{i}p_{i}}{i}\right)\right)\right).
 \end{multline*}

\noindent The plethysm also affects the long and short brackets and is now defined as:

$$Z_{M \circ N} = Z_{M} \ast Z_{N} = Z_{M}(u;w;q;p_{i} \mapsto p_{i} \ast Z_{N}),$$

\noindent where

$$p_{i}\ast Z_{N} = Z_{N}(u \mapsto u^{i}; \; w \mapsto w^{i}; \; q \mapsto (-1)^{i-1}q^{i};\; p_{j}\mapsto p_{ij}).$$

\begin{theorem}
For $k\geq 3$ and $d\geq 2$,
\begin{multline}
$$Z_{H_{\ast}\mathcal{M}_{d}^{(k)}}(u,w,q;p_{1},p_{2},p_{3},...) = \\
e^{(\sum_{l=1}^{\infty}\frac{p_l}{l})}\prod_{m=1}^{\infty}\Bigg(1-\frac{w^{m}}{u^{m(k-1)}}\Bigg[(-q)^{m(k-2)} - \\ (-q)^{m(k-2)} \left(e^{-\sum_{j=1}^{\infty}\frac{(-1)^{d-1}(-q)^{mj(d-1)}u^{mj}p_{mj}}{j}}\right)_{\leq m(k-1)}\left(e^{\sum_{j=1}^{\infty}\frac{(-1)^{d-1}(-q)^{mj(d-1)}u^{mj}p_{mj}}{j}}\right)\Bigg]\Bigg)^{(-1)^{d}E_{m}\left(\frac{1}{(-q)^{d-1}u}\right)}.$$ \tag{5.1} \label{refinement}
\end{multline}
\end{theorem}
\noindent Note that for $k=2$ or $d=1$ we do not get a splitting but rather a filtration. The formula \eqref{refinement} can still be applied, and it computes the cycle index sum of the symmetric group action on the associated graded factor~\cite{D_T}.

The proof of Theorem 2 follows the same steps as the proof of Theorem 1.

\nocite{*}
\bibliography{CycleIndexSumBib} 
\bibliographystyle{plain}

\end{document}